\documentclass{amsart}

\usepackage{amsfonts}       
\usepackage{amsthm}
\usepackage{amssymb}
\usepackage{graphicx}

\newtheorem{theorem}{Theorem}[section]
\newtheorem{lemma}[theorem]{Lemma}
\newtheorem{corollary}[theorem]{Corollary}
\newtheorem{proposition}[theorem]{Proposition}

\theoremstyle{definition}
\newtheorem{definition}[theorem]{Definition}
\newtheorem{example}[theorem]{Example}

\theoremstyle{remark}
\newtheorem*{remark}{Remark}

\numberwithin{equation}{section}

\begin{document}

\title{Geometric Hardy Inequalities via Integration on Flows}

\author{M. Paschalis}
\address{Department of Mathematics, University of Athens}

\email{mpaschal@math.uoa.gr}

\subjclass[2000]{Primary 35J75, 58J05; Secondary 58J05, 35R01, 35J62}


\keywords{Hardy Inequality; Directional derivative; Vector field; Flow; Manifold.}

\begin{abstract}
We introduce a geometric approach of integration along integral curves for functional inequalities involving directional derivatives in the general context of differentiable manifolds that are equipped with a volume form. We focus on Hardy-type inequalities and the explicit optimal Hardy potentials that are induced by this method. We then apply the method to retrieve some known inequalities and establish some new ones.
\end{abstract}

\maketitle

\section{Introduction}

We begin by providing some background on Hardy inequalities. The classic $L^p$ Hardy inequality in $\mathbb{R}^N$ reads $$\int_{\mathbb{R}^N} |\nabla \varphi(x)|^p dx \geq \bigg| \frac{p-N}{p} \bigg|^p \int_{\mathbb{R}^N} \frac{|\varphi(x)|^p}{|x|^p} dx, \ \ \ \varphi \in C^1_c(\mathbb{R}^N \setminus \{ 0 \}),$$ and has important applications in the theory of PDEs involving singular potentials. Without surprise, the Euclidean case of this type of inequality and its spin-offs has been studied extensively for several decades. For an extensive reference on Hardy inequalities, see \cite{BEL}.

A type of Hardy inequality that has attracted a lot of interest lately is one that involves the distance from the boundary of a domain $U\subset \mathbb{R}^N$. In particular, such a result should read $$\int_U |\nabla \varphi(x)|^p dx \geq c \int_U \frac{|\varphi(x)|^p}{d(x)^p} dx, \ \ \ \varphi \in C^1_c(U),$$ where $d(x) = \textrm{dist}(x, \partial U)$ and $c$ is the optimal positive constant (if any) for which the inequality is valid. Such results are known to exist if $U$ is convex or when $U$ is bounded and has Lipschitz boundary, for example. In particular, the one-dimensional case for the interval $(a,b)$ reads $$(\star) \ \ \ \int_a^b |\varphi'(x)|^p dx \geq \bigg( \frac{p-1}{p} \bigg)^p \int_a^b \frac{|\varphi(x)|^p}{\min \{ x-a, b-x \}^p} dx, \ \ \ \varphi\in C^1_c((a,b)).$$

In the present article, we propose a method of integration along integral curves to obtain a ``lifting" of this inequality for differentiable manifolds of arbitrary dimension that are subject to a simple geometric condition that is satisfied in a large number of cases. In particular, if $M$ is an oriented differentiable manifold with positive volume form $\omega$ and $X$ is a non-vanishing vector field on $M$, we prove the optimal inequality $$\int_M |X\varphi|^p \omega \geq \bigg( \frac{p-1}{p} \bigg)^p \int_M \frac{|\varphi|^p}{\tau_p^p} \omega, \ \ \ \varphi \in C^1_c(M),$$ where $\tau_p$ is a suitable ``boundary distance" that depends on the geometry of the configuration. It is worth noting that in our method $\tau_p$ is calculated explicitly and is usually highly non-trivial, except for the simplest of cases.

It has been recently pointed out to us by Y. Pinchover that a special case of this approach also appears in \cite{MMP}, where the authors integrate with respect to ``flow coordinates'' in bounded $C^2$ Euclidean domains to specify some properties of the Hardy constant that corresponds to the Euclidean distance, amongst other things. In this respect, our work could be considered to be a generalisation of this methodology in a broader context.

Although our results apply more generally, of special interest is the case of a Riemannian manifold $(M,g)$, where we can apply the method to retrieve inequalities involving the Riemannian gradient $\nabla_g$ and the associated volume form $\omega_g$. Our method can easily provide optimal, non-trivial Hardy potentials in a multitude of such cases, as we demonstrate through specific examples.

\section{Preliminaries}

We begin by setting the context and introducing the necessary notions that will be used throughout the rest of this work. 

\begin{definition}
Let $M$ be a smooth manifold of dimension $N$.
\begin{enumerate}
\item A non-vanishing vector field $X\in \Gamma(TM)$ is called a \textit{direction field} on $M$. The pair $(M,X)$ is then called a \textit{directed space}.

\item A non-vanishing $N$-form $\omega\in \Lambda^N(T^*M)$ is called a \textit{volume form} on $M$.

\item A triple $(M,X,\omega)$ that consists of a smooth manifold, a direction field and a volume form is called a \textit{directed volume space}.
\end{enumerate}
\end{definition}

In what follows and unless otherwise stated, $M$ will stand for a non-compact, oriented smooth manifold of dimension $N$, $X$ will be a direction field and $\omega$ will be a volume form on $M$. Hereafter, we will also make the implicit assumption that $\omega$ is positive in the chosen orientation.

As usual, an integral curve on the directed space $(M,X)$ will be a curve $\gamma:I\rightarrow M$ such that $\gamma' = X \circ \gamma$. By the existence and uniqueness theorem for ODEs, for each point $z\in M$, there exists a unique maximal integral curve $\gamma_z:I_z\rightarrow M$ such that $\gamma_z(0)=z$. The flow of $X$ is then defined to be the smooth map $$\theta: \bigsqcup_{z \in M} I_z \rightarrow M, \ \theta(z,t) = \gamma_z(t).$$

The directed space $(M,X)$ is said to be complete if $I_z=\mathbb{R}$ for all $z\in M$. The type of spaces that will occupy our attention are essentially the opposite of complete spaces in the following sense.

\begin{definition}
A directed space is said to be \textit{traceable} if $I_z \subsetneqq \mathbb{R}$ for all $z\in M$.
\end{definition}

To get an intuitive understanding of this definition, consider the one-point compactification of $M$ with $\infty$ being the point at infinity. Traceable spaces are exactly the ones in which starting at any point and following the flow of the field will take one to $\infty$ at finite time in at least one direction (positive or negative time).

Traceable spaces are important for our purposes because one can naturally define a temporal distance function from infinity: if $z\in M$ is a point, define $$\tau(z)=\textrm{dist}(0,\partial I_z).$$ Then $\tau: M \rightarrow \mathbb{R}$ is obviously well-defined and positive everywhere in the manifold, and its value at any point is equal to the time required to reach infinity if one follows the flow of the field starting from that point.

Each directed space $(M,X)$ comes naturally equipped with an equivalence relation $\sim$ that takes two points to be equivalent if they belong to the same integral curve. The resulting quotient space, which we denote by $M/X$, is called the \textit{orbit space} of $(M,X)$, and in general fails to be a manifold. We will be interested in subsets of $M$ that are saturated with respect to this relation.

\begin{definition}
Let $(M,X)$ be a directed space.
\begin{enumerate}
\item A subset $S \subset M$ is said to be \textit{saturated} if $\textrm{Im} (\gamma_z) \subset S$ for all $z\in S$.
\item If $S\subset M$ is any subset, we define the \textit{saturation of $S$} to be the set $$\theta(S)= \bigcup_{z\in S} \textrm{Im}(\gamma_z).$$
\end{enumerate}
\end{definition}
In other words, if a saturated subset $S$ contains a point then it contains the entire integral curve that point belongs to. Obviously, $S$ is saturated if and only if $S=\theta(S)$. Moreover, since the flow is an open map, if $S$ is open, so is $\theta(S)$.

In each directed space, one can introduce, at least locally, a set of normal coordinates $\chi=(t,s)=(t,s^1,\ldots,s^{N-1})$ with the property $\partial / \partial t = X$. In terms of the corresponding parametrisation $\zeta = \chi^{-1}$, this can be expressed equivalently as $$\partial_t \zeta(t,s) = X \circ \zeta(t,s).$$ Actually, this means that $\zeta$ forms a family of integral curves parametrised by $s$. While it is incorrect to assume that every directed space can be covered by a single normal coordinate chart, it is obvious that one always has an open cover of the manifold consisting of saturated normal chart domains (to see this, for each point $z\in M$, pick a normal coordinate ball $B$ centered at $z$ and consider $\theta(B)$).

In normal coordinates, $\omega$ admits a local expression $$\omega = \Omega(t,s)dt\wedge ds,$$ with $\Omega$ being the \textit{local volume density} in these coordinates. In general, $\Omega$ depends both on $s$ and $t$. Directed volume spaces in which $\Omega$'s don't depend on $t$ form a special class which is much easier to deal with for our purposes, so we give them a name.

\begin{definition}
A directed volume space $(M,X,\omega)$ is called \textit{simple} if the local volume density of $\omega$ in normal coordinates is independent of $t$.
\end{definition}

We will develop a method of obtaining Hardy inequalities for directed volume spaces regardless of whether they are simple or not. In fact, the most interesting cases are usually non-simple. However, simple spaces, as we will see shortly, are much easier to deal with and are the natural starting point for our line of work.

\section{The simple case}

First we deal with simple spaces. The derivation of a Hardy inequality is much simpler in that case, and sets the background for the more advanced techniques that are required to treat the general case.

Intuitively, the method we develop can be described as follows:
\begin{enumerate}
\item Cover the space with saturated normal coordinate charts. This way we can ``write down" the space as a parametrised family of integral curves.
\item Apply the one-dimensional Hardy inequality $(\star)$ along each curve separately.
\item Integrate over all integral curves using the normal coordinates.
\end{enumerate}

At this point, we are ready to state and prove the main theorem of this section.

\begin{theorem}
\label{Th1}
Let $(M,X,\omega)$ be a simple and traceable directed volume space. Then the inequality
\begin{equation}
\int_M |X\varphi|^p \omega \geq \bigg( \frac{p-1}{p} \bigg)^p \int_M \frac{|\varphi|^p}{\tau^p} \omega, \ \ \ \varphi \in C^1_c(M)
\label{H1}
\end{equation}
holds for all $p>1$.
\end{theorem}

\begin{proof}
Let $(U,\chi)$ be a saturated normal coordinate chart on $M$ with $\chi: U \rightarrow \tilde{U}$ for some open $\tilde{U}\subseteq \mathbb{R}^{N-1}$ and let $\zeta=\chi^{-1}$ be the corresponding parametrisation. Since $U$ is saturated, $\tilde{U}$ must be of the form $\bigsqcup_{s\in S}I_s$ for some open $S\in \mathbb{R}^{N-1}$ and some intervals $I_s \subsetneqq \mathbb{R}$, so we have coordinates $(t,s)$ where $s\in S$ and $t\in I_s$. Let $$\omega = \Omega(s)dt \wedge ds$$ in these coordinates. Moreover, we clearly have that $$\zeta(t,s)=\gamma_{\zeta(0,s)}(t), \ \ \ t\in I_s = I_{\zeta(0,s)}$$ and that $$\tau \circ \zeta(t,s) = \textrm{dist}(t,\partial I_s).$$

Now, it is clear that $\varphi \circ \zeta(\cdot,s) \in C^1_c(I_s)$ for all $s\in S$. Applying the one-dimensional Hardy inequality $(\star)$ on $\varphi \circ \zeta(\cdot,s)$ for fixed $s$ we get $$\int_{I_s} |\partial_t (\varphi \circ \zeta)(t,s))|^p dt \geq \bigg( \frac{p-1}{p} \bigg)^p \int_{I_s} \frac{|\varphi \circ \zeta(t,s)|^p}{\textrm{dist}(t,\partial I_s)^p} dt,$$ which by the properties of normal coordinates is equivalent to $$\int_{I_s} |X \varphi \circ \zeta(t,s)|^p dt \geq \bigg( \frac{p-1}{p} \bigg)^p \int_{I_s} \frac{|\varphi \circ \zeta(t,s)|^p}{\tau^p \circ \zeta(t,s)} dt.$$ Multiplying both sides by $\Omega(s)$ (which is positive by assumption), integrating over $S$ and applying Fubini's theorem yields $$\int_S \int_{I_s} |X \varphi \circ \zeta(t,s)|^p \Omega(s) dtds \geq \bigg( \frac{p-1}{p} \bigg)^p \int_S \int_{I_s} \frac{|\varphi \circ \zeta(t,s)|^p}{\tau^p \circ \zeta(t,s)} \Omega(s) dtds,$$ which, in terms of differential forms, is the same as $$\int_{\tilde{U}} |X \varphi \circ \zeta|^p \Omega \det \geq \bigg( \frac{p-1}{p} \bigg)^p \int_{\tilde{U}} \frac{|\varphi \circ \zeta|^p}{\tau^p \circ \zeta} \Omega \det.$$ The diffeomorphic invariance formula for integration on forms (see the Appendix) then yields $$\int_U |X\varphi|^p \omega \geq \bigg( \frac{p-1}{p} \bigg)^p \int_U \frac{|\varphi|^p}{\tau^p} \omega.$$

To complete the proof, let $\{(U_j,\chi_j)\}_{j\in J}$ be an atlas of $M$ that consists of saturated normal charts as above. The collection $\{U_j\}_{j\in J}$ is then an open cover of $M$, and therefore an open cover of $\textrm{supp}(\varphi)$. Furthermore, $\textrm{supp}(\varphi)$, being compact, must have a finite subcover $\{ U_1,\ldots,U_n \}$. For the final step, consider the saturated open sets $W_1,\ldots,W_n$, defined as $$W_1=U_1, \ \ \ W_k= U_k \setminus \bigcup_{l=1}^{k-1} \bar{U}_l.$$ The collection $\{ (W_k,\chi_k) \}$ and its corresponding parametrisations then satisfy the conditions of Lemma A.2 and the proof is finished.
\end{proof}

In some cases, the last argument can be replaced by a partition of unity argument. This would require that we project an open cover onto $M/X$, and then assume a partition of unity for the projected cover. However, this assumption is not always valid, as $M/X$ need not be Hausdorff.

Another, more important point is to note that the constant that appears in the theorem is optimal. Seeing that this is so is rather straightforward: simply pick a sequence $\varphi_\epsilon$ such that $\textrm{supp}(\varphi_\epsilon)$ converges to a single integral curve. If the inequality where to hold true for a larger constant, that would mean that the one-dimensional Hardy inequality from which it was derived would also hold for that constant, which is known to be false.

\begin{example}
The prototype of simple traceable spaces spaces is the Euclidean half-space $\mathbb{R}^N_+= \{ x\in \mathbb{R}^N : x_N>0 \} $ equipped with the parallel vector field $\partial/\partial x_N$. The normal coordinates in this case are given by $t=x_N$ and $s=(x_1,\ldots,x_{N-1})$, so we have that $\omega = dx_1 \wedge \cdots \wedge dx_N = dt \wedge ds$ (if necessary, take one of the s coordinates to have an opposite sign in order to mitigate the extra sign that might occur from changing the order in the exterior product). Moreover, we clearly have $\tau = x_N$, so it follows from Theorem 3.1. that the inequality $$\int_{\mathbb{R}^N_+} \bigg| \frac{\partial \varphi}{\partial x_N} \bigg|^p dx \geq \bigg( \frac{p-1}{p} \bigg)^p \int_{\mathbb{R}^N_+} \frac{|\varphi|^p}{x_N^p} dx, \ \ \ \varphi\in C^1_c(\mathbb{R}^N_+)$$ holds for all $p>1$.
\end{example}

\begin{example}
A less trivial example that still falls within the class of simple cases is that of a two-dimensional angle $A=\{ x\in \mathbb{R}^2 : 0<\theta(x)<\alpha \}$ (for some given $\alpha \in (0,2\pi])$ equipped with the vector field $$X=r^{\epsilon/p} \frac{\partial}{\partial \theta}$$ for some $p>1$ and some $\epsilon \in \mathbb{R}$. In polar coordinates, we have $\omega = rd\theta \wedge dr$. To find a set of normal coordinates $(t,s)$ for this configuration, choose $s=r$ and notice that we must wave $$\frac{\partial}{\partial t} = r^{\epsilon/p} \frac{\partial}{\partial \theta},$$ and therefore we may choose $t=\frac{\theta}{r^{\epsilon/p}}$. Moreover, it follows that $$d\theta = s^{\epsilon/p}dt +\frac{\epsilon}{p}ts^\frac{\epsilon-p}{p} ds,$$ hence $\omega = s^{\epsilon/p+1}dt\wedge ds$, so $(A,X,\omega)$ is simple. Since the integral curves here follow co-centric circles each with angular velocity $r^{\epsilon/p}$, it follows that $\tau = r^{-\epsilon/p} \min \{\theta, \alpha - \theta \}$. Direct application of Theorem \ref{H1} yields the inequality $$\int_A r^\epsilon \bigg| \frac{\partial \varphi}{\partial \theta} \bigg|^p dx \geq \bigg( \frac{p-1}{p} \bigg)^p \int_A r^\epsilon \frac{|\varphi|^p}{\min \{\theta, \alpha - \theta \}^p} dx, \ \ \ \varphi \in C^1_c(A).$$ It is worth noting that in the special case $\epsilon=-p$, we get an inequality involving the angular component of the gradient, thus we have $$\int_A |\nabla \varphi|^p dx \geq \bigg( \frac{p-1}{p} \bigg)^p \int_A \frac{|\varphi|^p}{r^p \min \{\theta, \alpha - \theta \}^p} dx, \ \ \ \varphi \in C^1_c(A).$$
\end{example}

\section{$p$-normal coordinates}

The proof of (\ref{H1}) was based on the fact that we can multiply the integral over $dt$ with $\Omega(s)$ and then pass $\Omega(s)$ inside the integral (since it is independent of $t$). If we look at the more general case of a non-simple space where $\Omega(t,s)$ depends also on $t$, it is clear that one cannot repeat this argument.

We can bypass this difficulty by introducing new coordinates that are related to the initial set of normal coordinates $(t,s)$. These new coordinates, denoted $(t',s')$, will have the property $$X = \frac{\partial}{\partial t} = \Omega'(t',s')^{-1/p} \frac{\partial}{\partial t'},$$ where $\Omega'(t',s') = \omega(\partial_{t'},\partial_{s'})$ is the local volume density in these new coordinates. This way, we can get an integral over $dt'$ which contains both the correct vector field and the correct volume element from the beginning.

This motivates the following definition.

\begin{definition}
Let $(M,X,\omega)$ be a directed volume space, and let $p>1$. A set of coordinates $(\tau,\sigma)=(\tau,\sigma^1, \ldots ,\sigma^{N-1})$ (defined on some open set) will be called a set of \textit{$p$-normal coordinates along $X$ with respect to $\omega$}  if $$X=\Omega(\tau,\sigma)^{-1/p} \frac{\partial}{\partial \tau}.$$
\end{definition}

We dedicate the remainder of this section to prove the existence and some useful properties of these coordinates. We also explore their connection to regular normal coordinates as defined previously, and relate to them a well-defined (independent of coordinates) temporal/volumetric ``distance" like $\tau$ in the previous sections. These facts will form the necessary background to generalise Theorem \ref{Th1} to include non-simple spaces.

\begin{proposition}[Existence]
Let $(t,s)$ be a set of normal coordinates on some open $U\subseteq M$ in the directed volume space $(M,X,\omega)$. The coordinates $(t',s')$ defined by $$t'=\int^t \Omega(\xi,s)^{-\frac{1}{p-1}} d\xi, \ \ \ s'=s$$ is a set of $p$-normal coordinates along $X$ with respect to $\omega$ on $U$.
\end{proposition}

\begin{proof}
It is clear that $$\frac{\partial t'}{\partial t} = \Omega(t,s)^{-\frac{1}{p-1}},$$ and we calculate $$\Omega(t,s)=\omega(\partial_t,\partial_s)=\frac{\partial t'}{\partial t} \omega (\partial_{t'},\partial_{s'})= \frac{\partial t'}{\partial t} \Omega'(t',s').$$ It follows that $$\frac{\partial t'}{\partial t} = \Omega'(t',s')^{-1/p},$$ thus $$\Omega'(t',s')^{-1/p} \frac{\partial}{\partial t'} = \frac{\partial}{\partial t} = X,$$ so the set of coordinates $(t',s')$ is indeed $p$-normal along $X$ with respect to $\omega$.

Since $\omega$ is non-vanishing, it follows that $\Omega(t,s) >0$, so in particular $t'$ is well-defined everywhere in $U$.
\end{proof}

This not only proves existence, but also provides a practical way to compute such coordinates, provided we already have a set of normal coordinates, which are often straightforward to acquire.

Another fact is that these coordinates cooperate well with the flow of the field $X$. If we choose a saturated normal chart, which we already know how to produce, it is straightforward to turn it into a $p$-normal saturated coordinate chart using the above transformation. This is evident from the fact that the vector field $\Omega(\tau,\sigma)^{-1/p}X$ has the same integral curves as $X$, only reparametrised.

Recall that for a directed volume space, we defined the associated temporal distance $\tau:M \rightarrow \mathbb{R}$, which essentially measures the amount of time required to reach the ``boundary" of $M$ moving along the flow of $X$. Equivalently, if $(t,s)$ is a set of normal coordinates in a saturated domain such that $t\in I_s = (a_s,b_s)$, then $$\tau= \textrm{dist}(t,\partial I_s)= \min (t-a_s,b_s-t).$$ We now introduce the following notation.

\begin{definition}
Let $f:I\rightarrow \mathbb{R}$ be a measurable function on the interval $I=(a,b)$ (here it is possible that $a=-\infty$ or $b=+\infty$). Define $$\int_{\partial I}^t f(\xi)d\xi = \min \bigg( \int_a^t f(\xi)d\xi,\int_t^b f(\xi)d\xi \bigg).$$
\end{definition}

In this notation, it is clear that $$\tau=\int_{\partial I_s}^t d\xi.$$ Moreover, the condition that $(M,X)$ is traceable can be rewritten as $$\tau=\int_{\partial I_s}^t d\xi < \infty \textrm{ everywhere in } M.$$

It turns out that what we need in the case of non-simple spaces, is a suitable modification of this with respect to $p$-normal coordinates.

\begin{definition}
Let $(M,X,\omega)$ be a directed volume space and let $(t,s)$, $t\in I_s$ be normal coordinates for a saturated chart domain $U\subseteq M$, let $\Omega(t,s)$ be the local volume density in these coordinates and let $p>1$.
\begin{enumerate}
\item We say that $U$ is \textit{$p$-traceable} if $$\int_{\partial I_s}^t \Omega(\xi,s)^{-\frac{1}{p-1}} d\xi < \infty \textrm{ everywhere in } U.$$
\item If $U$ is $p$-traceable, we define the associated \textit{temporal/volumetric distance} $\tau_p:U\rightarrow \mathbb{R}$ to be the function $$\tau_p = \Omega(t,s)^\frac{1}{p-1} \int_{\partial I_s}^t \Omega(\xi,s)^{-\frac{1}{p-1}} d\xi.$$
\end{enumerate}
\end{definition}

\begin{proposition}
Everything in the above definition is well-defined, i.e. independent of the choice of normal coordinates in $U$.
\end{proposition}

\begin{proof}
Suppose that we have two sets of normal coordinates $(t,s)=(t,s^1,\ldots, s^{N-1})$ and $(t',s')=(t',(s')^1,\ldots,(s')^{N-1})$ of the same orientation in U. By the chain rule, we have that
$$\frac{\partial}{\partial t} = \frac{\partial t'}{\partial t} \frac{\partial}{\partial t'} + \sum_{j=1}^{N-1} \frac{\partial (s')^j}{\partial t} \frac{\partial}{\partial (s')^j},$$ $$\frac{\partial}{\partial s^i} = \frac{\partial t'}{\partial s^i} \frac{\partial}{\partial t'} + \sum_{j=1}^{N-1} \frac{\partial (s')^{j}}{\partial s^i} \frac{\partial}{\partial (s')^{j}},$$ where $i=1,\ldots, N-1$. Since these are both sets of normal coordinates, we must have $\partial_t = \partial_{t'} = X$. This implies that $$\frac{\partial t'}{\partial t} = 1 \textrm{ and } \frac{\partial (s')^j}{\partial t} = 0 \textrm{ for all } j=1,\ldots , N-1.$$ In particular, the $s'$ coordinates are independent of $t$ and $s'=\sigma(s)$ for some diffeomorphism $\sigma$ between open sets in $\mathbb{R}^{N-1}$.

By linearity and skew-symmetry of $\omega$, we have that $$\Omega(t,s)=\omega \bigg( \frac{\partial}{\partial t}, \frac{\partial}{\partial s^1},\ldots,\frac{\partial}{\partial s^{N-1}} \bigg)=$$ $$\sum_{j_1,\ldots,j_{N-1}=1}^{N-1} \frac{\partial (s')^{j_1}}{\partial s^1} \cdots \frac{\partial (s')^{j_{N-1}}}{\partial s^{N-1}} \omega \bigg( \frac{\partial}{\partial t'}, \frac{\partial}{\partial (s')^{j_1}},\ldots,\frac{\partial}{\partial (s')^{j_{N-1}}} \bigg)=$$ $$\sum_{\pi \in S_{N-1}} \frac{\partial (s')^{\pi(1)}}{\partial s^1} \cdots \frac{\partial (s')^{\pi(N-1)}}{\partial s^{N-1}} (-1)^\pi \Omega'(t',s'),$$ where the last sum is over all permutations $\pi$ in $(N-1)$ elements and $(-1)^\pi$ is the sign of $\pi$. It follows that $$\frac{\Omega(t,s)}{\det D\sigma(s)} = \Omega'(t',s'),$$ where $D\sigma$ is the Jacobian matrix of $s'=\sigma(s)$. Since $\sigma$ is an orientation-preserving diffeomorphism, this matrix is non-singular and the determinant is positive.

It is straightforward to show that neither the convergence of the integral in (1) of the definition nor the formula of $\tau_p$ in (2) are affected if we switch between normal coordinates. Indeed, we have that $$\int_{\partial I_{s'}}^{t'} \Omega'(\xi',s')^{-\frac{1}{p-1}} d\xi' = \int_{\partial I_s}^t \bigg[ \frac{\Omega(\xi,s)}{\det D\sigma(s)} \bigg]^{-\frac{1}{p-1}} \frac{d\xi'}{d\xi} d\xi =$$ $$(\det D\sigma(s))^{\frac{1}{p-1}} \int_{\partial I_s}^t \Omega(\xi,s)^{-\frac{1}{p-1}} d\xi,$$ so $$\int_{\partial I_s}^t \Omega(\xi,s)^{-\frac{1}{p-1}} d\xi <\infty \Leftrightarrow \int_{\partial I_{s'}}^{t'} \Omega'(\xi',s')^{-\frac{1}{p-1}} d\xi'<\infty,$$ and it is clear that $\tau_p=\tau_p'$.
\end{proof}

Since every directed space $(M,X)$ admits an open cover of saturated normal coordinate charts, and since the above notions are independent of the choice of such a chart, we can unambiguously extend these notions over the whole manifold. This way we may define the global function $\tau_p: M \rightarrow \mathbb{R}$ given locally by $$\tau_p = \Omega(t,s)^\frac{1}{p-1} \int_{\partial I_s}^t \Omega(\xi,s)^{-\frac{1}{p-1}} d\xi.$$ At this point, it is clear that $(M,X,\omega)$ is $p$-traceable if and only if the function $\tau_p$ is defined everywhere in $M$.

As a final remark, we would like to point out that in the case where $(M,X,\omega)$ is simple, $p$-traceability coincides with traceability and $\tau_p =\tau$, so this is indeed a meaningful extension of the previous concepts.

\section{The general case}

We are now ready to state and prove our main result.

\begin{theorem}
\label{Th2}
Let $(M,X,\omega)$ be a directed volume space and let $p>1$. Then the inequality
\begin{equation}
\label{H2}
\int_M |X\varphi|^p \omega \geq \bigg( \frac{p-1}{p} \bigg)^p \int_M \frac{|\varphi|^p}{\tau_p^p} \omega,\ \ \ \varphi \in C^1_c(M)
\end{equation}
is valid whenever the space is $p$-traceable.
\end{theorem}

\begin{proof}
Let $U$ be a saturated coordinate domain with normal coordinates $\chi=\zeta^{-1}=(t,s)$ and corresponding $p$-normal coordinates $\chi'=(\zeta')^{-1}=(t',s')$ constructed as demonstrated in the previous section. Let $\varphi \in C^1_c(M)$. As with the simple case, apply the one-dimensional Hardy inequality to $\varphi \circ \zeta'(\cdot,s) \in C^1_c(I_{s'})$ to get $$\int_{I_{s'}} |\partial_{t'} (\varphi \circ \zeta')(t',s')|^p dt' \geq \bigg( \frac{p-1}{p} \bigg)^p \int_{I_{s'}} \frac{| \varphi \circ \zeta' (t',s')|^p}{\textrm{dist}^p(t', \partial I_{s'})} dt',$$ which by the properties of the $p$-normal coordinates becomes $$\int_{I_{s'}} |X \varphi \circ \zeta' (t',s')|^p \Omega'(t',s')dt' \geq \bigg( \frac{p-1}{p} \bigg)^p \int_{I_{s'}} \frac{| \varphi \circ \zeta' (t',s')|^p}{\textrm{dist}^p(t', \partial I_{s'})} dt'.$$ Integrating both sides over the $s'$-coordinates then yields $$\int_{S'} \int_{I_{s'}} |X \varphi \circ \zeta' (t',s')|^p \Omega'(t',s') dt'ds' \geq $$ $$\bigg( \frac{p-1}{p} \bigg)^p \int_{S'} \int_{I_{s'}} \frac{| \varphi \circ \zeta' (t',s')|^p}{\Omega'(t',s')\textrm{dist}^p(t', \partial I_{s'})} \Omega'(t',s') dt'ds'.$$

To show that this is the same as $$\int_U |X\varphi|^p \omega \geq \bigg( \frac{p-1}{p} \bigg)^p \int_U \frac{|\varphi|^p}{\tau_p^p} \omega,$$ it remains to be shown that $\tau_p^p= \Omega'(t',s') \textrm{dist}^p(t',\partial I_{s'})$. This is straightforward, as we have $$\textrm{dist}(t',\partial I_{s'}) = \int_{\partial I_s}^t \Omega^{-\frac{1}{p-1}}(\xi,s)d\xi$$ from the definition, and by elementary calculations we also have that $\Omega'(t',s')=\Omega^\frac{p}{p-1}(t,s)$.

The proof is again completed by a similar argument as in \ref{Th1}.
\end{proof}

Let us make a few remarks about the result. The first is its generality. The only condition that we have imposed for the inequality to hold true is $p$-traceability of $(M,X,\omega)$. The number of cases this applies to is vast, including many important cases that are already of interest. We will provide specific examples in the remainder of our work. For the time being, let us note that the only thing we need - in principle - in order to check whether the condition is satisfied is to find a set of normal coordinates $(t,s)$, compute the local volume density $\Omega(t,s)$ in these coordinates and then check if the integral $$\int_{\partial I_s}^t \Omega^{-\frac{1}{p-1}}(\xi,s) d\xi$$ converges. In a large number of cases, including many of the cases that are of immediate interest, this poses no real hardship.

What we gain from this process, however, is often highly non-trivial results. If the space in question indeed turns out to be $p$-traceable, the result provides an explicit, optimal Hardy potential in terms of the induced temporal/volumetric distance $$\Omega^{\frac{1}{p-1}}(t,s)\int_{\partial I_s}^t \Omega^{-\frac{1}{p-1}}(\xi,s) d\xi.$$

\begin{example}
As an elementary application to showcase how the method works in practice, we provide an alternative proof of the standard Euclidean Hardy inequality in $\mathbb{R}^N$ featuring the distance from a single point. Here, choose $M= \mathbb{R}^N \setminus \{ 0 \}$, $X=\partial /\partial r$ and $\omega = \det$ (the Euclidean volume form).

Finding normal coordinates for this configuration is trivial: since we must have $$\frac{\partial}{\partial t} = \frac{\partial}{\partial r},$$ simply choose $t=r$. For the rest of the coordinates there is a lot of freedom of choice, but we can simply choose $s=\theta$, where $\theta$ are the angles in the spherical coordinate system (therefore the spherical coordinates as a whole forms a set of normal coordinates in our case).

The expression of the Euclidean volume form in spherical coordinates is of the form $\omega = r^{N-1}f(\theta) dr \wedge d\theta$ for some $f(\theta)$ that involves powers  of sines of the angles, therefore in our chosen normal coordinates we have the same representation $$\omega = t^{N-1}f(s)dt \wedge ds,$$ so it is clear that the local volume density is $\Omega(t,s) = t^{N-1}f(s)$.

Now let $1<p \neq N$. The temporal/volumetric distance is $$\tau_p = \Omega^{\frac{1}{p-1}}(t,s)\int_{\partial I_s}^t \Omega^{-\frac{1}{p-1}}(\xi,s) d\xi = t^{\frac{N-1}{p-1}} \int_{\{ 0,\infty \}}^t \xi^{-\frac{N-1}{p-1}} d\xi.$$ To compute this, we must consider the two different cases $p<N$ and $p>N$, but in either case the result is $$\tau_p=\frac{p-1}{|p-N|}r.$$ By \ref{Th2}, it follows that the inequality $$\int_{\mathbb{R}^N} \bigg| \frac{\partial \varphi}{\partial r} \bigg|^p dx \geq \bigg| \frac{p-N}{p} \bigg|^p \int_{\mathbb{R}^N} \frac{|\varphi|^p}{r^p} dx$$ holds for all $\varphi \in C^1_c(\mathbb{R}^N \setminus \{ 0 \})$, as expected.

However, notice the unorthodox manner in which we obtain the best constant. In our method, this constant is not merely the result of algebraic operations, but has a geometric significance as well: it is a direct consequence of the $p$-dependence of the distance $\tau_p$.
\end{example}

\begin{example}
In the same manner as in the previous example, by choosing $X=r^{-\epsilon /p} \partial / \partial r$ we can prove the weighted inequality $$\int_{\mathbb{R}^N} \frac{1}{r^\epsilon} \bigg| \frac{\partial \varphi}{\partial r} \bigg|^p dx \geq \bigg| \frac{p-N+\epsilon}{p} \bigg|^p \int_{\mathbb{R}^N} \frac{|\varphi|^p}{r^{p+\epsilon}} dx, \ \ \ \varphi \in C^1_c(\mathbb{R}^N \setminus \{ 0 \})$$ for $\epsilon \neq N-p$. The calculations are a bit more involved than before but still elementary.
\end{example}

\begin{example}
As a final example, we turn our attention to the hyperbolic space $\mathbb{H}^N$, where a peculiar phenomenon occurs: the Hardy inequality becomes a Poincar\'e inequality. We employ the Poincar\'e half space model, where $\mathbb{H}^n = \{ x\in \mathbb{R}^N : x_N>0 \}$ with $g_{\mathbb{H}^N} = \frac{1}{x_N^2} g_{\mathbb{R}^N}$. The Riemannian volume form in this case reads $\omega_{\mathbb{H}^N} = x_N^{-N} \det$. Let $X = x_N \frac{\partial}{\partial x_N}$. It is clear that $|X|=1$. To find a set of normal coordinates for $(\mathbb{H}^N,X)$ we must find a $t$ such that $$\frac{\partial}{\partial t} = x_N \frac{\partial}{\partial x_N},$$ so we choose $t = \log x_N$ and $s=(x_1,\ldots,x_{N-1})$. It follows that $\omega = e^{-(N-1)t} ds \wedge dt$. Finally, we calculate $$\tau_p = e^{-\frac{N-1}{p-1} t} \int_{-\infty}^t e^{\frac{N-1}{p-1} \xi} d\xi = \frac{p-1}{N-1},$$ from which we obtain the inequality $$\int_{\mathbb{H}^N} |\nabla_{\mathbb{H}^N} \varphi|^p \omega_{\mathbb{H}^N} \geq \int_{\mathbb{H}^N} \bigg| x_N \frac{\partial \varphi}{\partial x_N} \bigg|^p \omega_{\mathbb{H}^N} \geq \bigg( \frac{N-1}{p} \bigg)^p \int_{\mathbb{H}^N} |\varphi|^p \omega_{\mathbb{H}^N}.$$ This is the classic Poincar\'e inequality for the hyperbolic space, and it is already known to be a consequence of the Hardy inequality (it can actually be obtained via the weighted inequality of the previous example, with minor modifications).
\end{example}

At this point it becomes clear that, when referring to the temporal/volumetric distance, the word ``distance" should not be taken too literally, since it does not always conform to the way we know a distance should behave (e.g. in the last example it was a constant).

\section{Application I: The exterior of a ball}

We will now use the method to obtain some new results. We would like to point out that there are new things that can be said even in the Euclidean case. In this section we focus on the case where $M=E$ is the exterior of a Euclidean ball of dimension $N$.

\begin{theorem}[Hardy Inequality for the exterior of a ball]
Let $E=\{x\in \mathbb{R}^N : |x|>R \}$ be the exterior of the $N$-dimensional Euclidean ball of radius $R$. Then the inequalities $$\int_E \bigg| \frac{\partial \varphi}{\partial r} \bigg|^p dx \geq \bigg( \frac{p-N}{p} \bigg)^p \int_E \frac{|\varphi|^p}{r^{\frac{N-1}{p-1}p}(r^\frac{p-N}{p-1}-R^\frac{p-N}{p-1})^p} dx, \ \ \ p>N,$$ $$\int_E \bigg| \frac{\partial \varphi}{\partial r} \bigg|^N dx \geq \bigg( \frac{N-1}{N} \bigg)^N \int_E \frac{|\varphi|^N}{r^N \log(r/R)^N} dx,$$ $$\int_E \bigg| \frac{\partial \varphi}{\partial r} \bigg|^p dx \geq \bigg( \frac{N-p}{p} \bigg)^p \int_E \frac{|\varphi|^p}{r^{\frac{N-1}{p-1}p}\min \{r^\frac{p-N}{p-1}, R^\frac{p-N}{p-1}-r^\frac{p-N}{p-1} \}^p} dx, \ \ \ 1<p<N$$ hold for all $\varphi \in C^1_c(E)$.
\end{theorem}

\begin{proof}
Similar to the case of $\mathbb{R}^N$ with the distance from a single point, the spherical coordinate system is a set of normal coordinates. The only difference now is that $t=r$ ranges from $R$ to $\infty$. Thus, we have $$\tau_p = t^{\frac{N-1}{p-1}} \int_{\{ R,\infty \}}^t \xi^{-\frac{N-1}{p-1}} d\xi,$$ so in each individual case
\[ 
\tau_p= \left\{
\begin{array}{ll}
      \frac{p-1}{p-N}r^\frac{N-1}{p-1}(r^\frac{p-N}{p-1}-R^\frac{p-N}{p-1}), & p>N \\
      r \log(r/R), & p=N.\\
      \frac{p-1}{N-p}r^\frac{N-1}{p-1}\min\{r^\frac{p-N}{p-1}, R^\frac{p-N}{p-1}-r^\frac{p-N}{p-1}\}, & p<N \\
\end{array} 
\right. 
\]
and the result follows.
\end{proof}

This is a non-trivial result, although its derivation has been trivialised by the use of our method. Let us make a few remarks on it. Note that in the cases where $p \neq N$, for small $r-R$ we have $\tau_p \approx r-R$, whereas for large $r-R$ we have $\tau_p \approx \frac{p-1}{|p-N|} (r-R)$. This fits our intuition: when close to the ball the inequality must behave like the one involving the distance from a hyperplane, while for very large distances it must resemble the one involving the distance from a point. In essence, the induced distance $\tau_p$ forms a continuous transition between these two limit cases.

It is also of practical importance to compare $\tau_p$ with the Euclidean distance from the boundary $d=r-R$. This will yield inequalities for the classic Hardy potential $V=d^{-p}$. To our knowledge, the only known result in this direction is given by Avkhadiev and Makarov in \cite{AM} (see also \cite{GPP} for alternative proofs of this result). The result states that for every compact $U \subseteq \mathbb{R}^N$, the best constant in the Hardy inequality $$\int_{\mathbb{R}^N \setminus U} |\nabla \varphi|^p dx \geq c \int_{\mathbb{R}^N \setminus U} \frac{|\varphi|^p}{d^p} dx,\ \ \ \varphi \in C^1_c(\mathbb{R}^N\setminus U),$$ is $c=\big( \frac{p-N}{p} \big)^p$ in the case where $p>N$, which implies the optimal inequality $$\int_E |\nabla \varphi|^p \geq \bigg( \frac{p-N}{p} \bigg)^p \int_E \frac{|\varphi|^p}{d^p},\ \ \ \varphi \in C^1_c(E)$$ in the case of the exterior of a ball. In that case our method gives $$\tau_p=\frac{p-1}{p-N}r^\frac{N-1}{p-1}(r^\frac{p-N}{p-1}-R^\frac{p-N}{p-1}).$$ To specify the best constant $\kappa$ such that $$\frac{1}{\tau_p} \geq \frac{\kappa}{d},$$ we make a few observations. As we already noted, we have $\tau_p(r)\approx d(r)$ for $r$ close to $R$ and $\tau_p(r)\approx \frac{p-1}{p-N} d(r)$ for large $r$. More generally, the derivative of $\tau_p(r)$ is given by $$\tau_p'(r)=\frac{p-1}{p-N}-\frac{N-1}{p-N} \bigg( \frac{R}{r} \bigg)^{\frac{p-N}{p-1}},$$ which is a strictly increasing function of $r$. It follows that $$\frac{p-1}{p-N} d(r) = \sup_{y>R}(\tau_p'(y) d(r) > \tau_p(r)$$ so $$\frac{1}{\tau_p} > \frac{p-N}{p-1} \frac{1}{d}$$ and we retrieve the same best constant $c= \big( \frac{p-N}{p} \big)^p$. It follows that our method improves the result of \cite{AM} in the case where $U$ is a ball, in the sense that it provides a better distance for the same constant.

For the case $p<N$, we have the following comparison.

\begin{corollary}
Let $E=\{x\in \mathbb{R}^N : |x|>R \}$. Then the inequality $$\int_E |\nabla \varphi|^p dx \geq \bigg( \frac{N-p}{p} \bigg)^p \big( 1-2^{-\frac{p-1}{N-p}} \big)^p \int_E \frac{|\varphi|^p}{d^p} dx,\ \ \ \varphi \in C^1_c(E)$$ holds for all $p<N$.
\end{corollary}

\begin{proof}
In this case, it is $$\tau_p=\frac{p-1}{N-p}r^\frac{N-1}{p-1}\min\{r^\frac{p-N}{p-1}, R^\frac{p-N}{p-1}-r^\frac{p-N}{p-1}\}.$$ We put $a=2^\frac{p-1}{N-p} R$, which is the real number such that $$a^\frac{p-N}{p-1}=R^\frac{p-N}{p-1}-a^\frac{p-N}{p-1},$$ i.e. the point in which the branch transition occurs. It follows that
\[ 
\tau_p= \left\{
\begin{array}{ll}
      \frac{p-1}{N-p}r, & r\geq a \\
      \frac{p-1}{N-p}r^\frac{N-1}{p-1}(R^\frac{p-N}{p-1}-r^\frac{p-N}{p-1}), & r<a \\
\end{array} .
\right. 
\]
An elementary calculation reveals that the derivative of $\tau_p(r)$ for $r<a$ is $$\tau_p'(r)=\frac{N-1}{N-p} ( r/R)^\frac{N-p}{p-1} - \frac{p-1}{N-p},$$ which is strictly increasing, so in particular $\tau_p(r)$ is convex for $r<a$. By virtue of Jensen's inequality it follows that $$\tau_p(r) \leq A (r-R),\ \ \ R<r<a,$$ where $$A=\frac{\tau_p(a)-\tau_p(R)}{a-R} = \frac{p-1}{N-p} \frac{a}{a-R} = \frac{p-1}{N-p} \big( 1-2^{-\frac{p-1}{N-p}} \big)^{-1}.$$

As for the region $r \geq a$, we certainly have that $\tau_p(r) \leq A (r-R)$, since both functions are affine, share the same value at $a$ and $\frac{p-1}{N-p}<A$.

So in any case we have $$\frac{1}{\tau_p} \geq \frac{N-p}{p} \frac{1-2^{-\frac{p-1}{N-p}}}{d}$$ and the result follows.
\end{proof}

For the sake of clarity, we give some plots of the function $\tau_p$ for specific values of $N$ and $p$, plotted against the function $\frac{p-1}{|p-N|} (r-R)$ that we use when making the Euclidean comparison (see Figure 1 below). Other choices of $N$ and $p$ give qualitatively similar results. What really matters is whether $p<N$ or $p>N$.

\begin{figure}[h]
\centering
\includegraphics[scale=0.45]{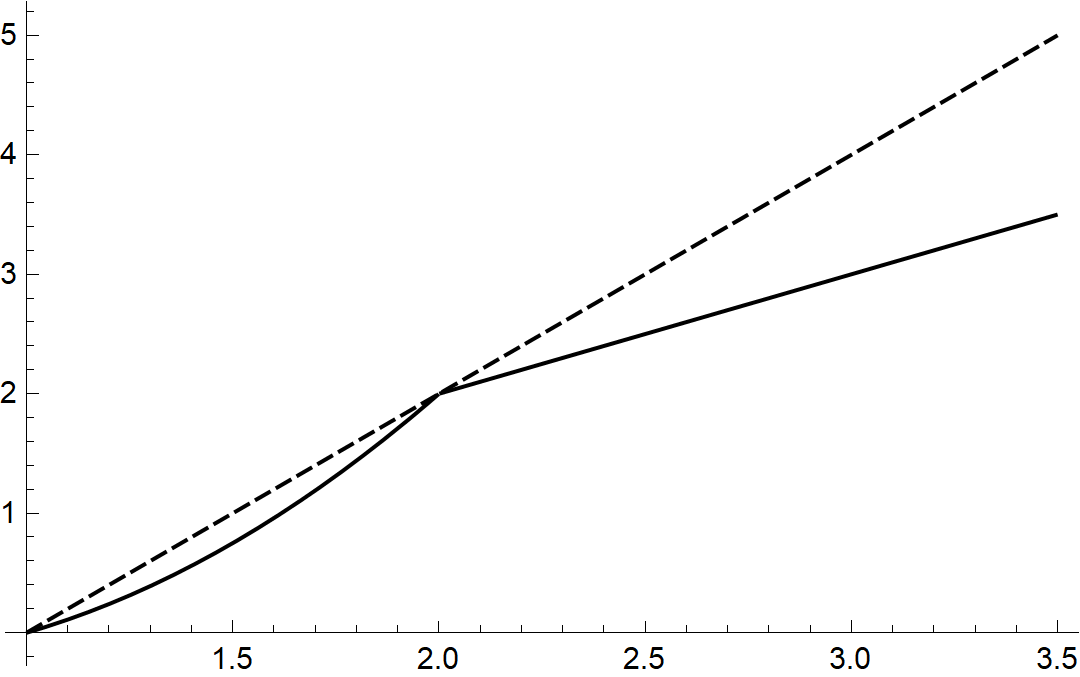}
\includegraphics[scale=0.45]{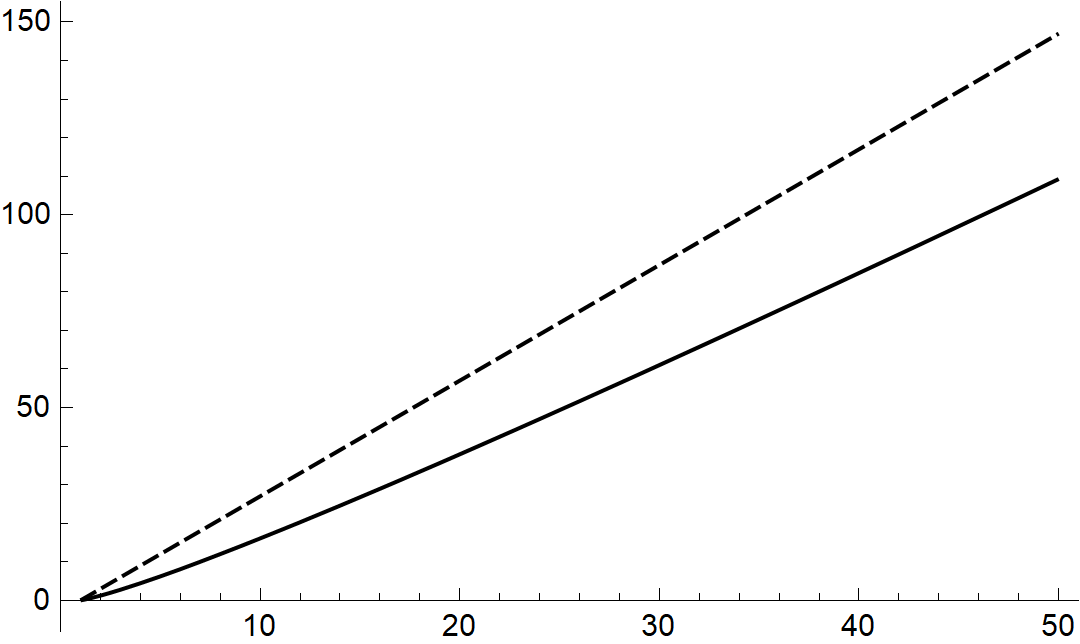}
\caption{Left: $R=1,N=3,p=2$. Right: $R=1,N=3,p=4$}
\end{figure}

\section{Application II: Spherical symmetry}

Moving beyond the classic Euclidean setting, the most important class of examples is arguably the class of spherically symmetric manifolds. We say that a Riemannian manifold $(M,g)$ is (locally) spherically symmetric around a central point $o\in M$ if the metric can be expressed as $$g=d\rho \otimes d\rho +\psi^2(\rho) g_{S^{N-1}}$$ in a punctured neighbourhood of $o$, where $\rho = \textrm{dist}(\cdot, o)$ is the Riemannian distance from $o$, $\psi$ is a positive function depending only on $\rho$ and $g_{S^{N-1}}$ is the round metric of the unit sphere of codimension 1. We are interested in the case where we have global spherical symmetry.

If $M$ is non-compact, the above polar representation extends to the whole punctured space $M'=M\setminus \{ o \}$. If $M$ is compact, we must exclude an additional ``antipodal" point $o' \in M$ (the most characteristic example is the sphere, where one must exclude both poles).

In either case, $\rho:M' \rightarrow \mathbb{R}$ has range of the form $(0,R)$ (we may have $R=+\infty$), and we may apply Theorem \ref{Th2} with $X=\partial / \partial \rho$ and $\omega = \omega_g = \psi^{N-1}(\rho) d\rho\wedge \omega_{S^{N-1}}$. In the following, we also take into account the case where we choose to exclude not only the ``pole(s)'' $o$ (and $o'$), but perhaps a larger object (for example, a geodesic ball around $o$ or $o'$).

\begin{theorem}
Suppose that $(M',g)$ is a Riemannian manifold whose metric can be expressed as $$g=d\rho \otimes d\rho +\psi^2(\rho) g_{S^{N-1}}$$ for some $\rho:M' \rightarrow (a,b)$ and some smooth $\psi:(a,b) \rightarrow (0,\infty)$. If for each value of $\rho \in (a,b)$, either one (or both) of the integrals $$\int_a^\rho \psi^{-\frac{N-1}{p-1}}(\xi)d\xi,\  \int_\rho^b \psi^{-\frac{N-1}{p-1}}(\xi)d\xi$$ converge, the inequality $$\int_{M'} |\partial_\rho \varphi|^p \omega_g \geq \bigg( \frac{p-1}{p} \bigg)^p \int_{M'} \frac{|\varphi|^p}{\varpi_p^p} \omega_g,\ \ \ \varphi \in C^1_c(M')$$ is valid with $$\varpi_p = \psi^{\frac{N-1}{p-1}}(\rho) \min \bigg( \int_a^\rho \psi^{-\frac{N-1}{p-1}}(\xi)d\xi,  \int_\rho^b \psi^{-\frac{N-1}{p-1}}(\xi)d\xi \bigg).$$
\end{theorem}

\begin{proof}
This is just a restatement of Theorem \ref{Th2} for the special case $(M,X,\omega) = (M',\partial_\rho,\omega_g)$.
\end{proof}

$M'$ can be thought of as a suitable open submanifold of a spherically symmetric manifold $M$. A key feature of our technique is that it effectively manages to take into account the volumetric/temporal distance from both the ``inner'' and the ``outer'' edge of the manifold. By ``inner'' edge we mean the edge that is closer to the central point $o$. The volumetric/temporal distance from the inner edge is given by $$\varpi_p^{\textrm{in}} = \psi^{\frac{N-1}{p-1}}(\rho)  \int_a^\rho \psi^{-\frac{N-1}{p-1}}(\xi)d\xi,$$ while the corresponding distance from the outer edge is $$\varpi_p^{\textrm{out}} = \psi^{\frac{N-1}{p-1}}(\rho) \int_\rho^b \psi^{-\frac{N-1}{p-1}}(\xi)d\xi.$$ While it is true that $\varpi_p = \min ( \varpi_p^{\textrm{in}}, \varpi_p^{\textrm{out}} )$, and consequently $$\frac{1}{\varpi_p} \geq \frac{1}{\varpi_p^{\textrm{in}}}, \frac{1}{\varpi_p^{\textrm{out}}},$$ it is sometimes convenient to consider Hardy potentials that take into account only the inner or outer edge. One may choose to do this in order to extend the class of admissible functions (in the case of a compact manifold where we have an antipodal point $o'$, one may still prefer to take into account functions that do not vanish at $o'$).

To this end, this is a good point to demonstrate the flexibility of our method: all that Theorem \ref{Th2} does is to essentially ``lift'' the one-dimensional Hardy inequality ($\star$) in higher dimensions.  As a matter of fact, \textit{any} one-dimensional functional inequality could be used in its place. Without straying from our subject of Hardy inequalities, we simply point out that one gets nearly identical results if we choose instead to lift the inequality $$(\star \star) \ \ \ \int_a^b |\varphi'(x)|^p dx \geq \bigg( \frac{p-1}{p} \bigg)^p \int_a^b \frac{|\varphi(x)|^p}{(x-a)^p} dx, \ \ \ \varphi\in C^1_c((a,b]),$$ which takes into account only the first endpoint and admissible functions need not vanish close to $b$. This gets us exactly what we need.

\begin{theorem}
Let $(M,g)$ be a compact, spherically symmetric manifold with empty boundary, with central point $o \in M$ of injectivity radius $\textrm{inj}(o)=R$, $\rho = \textrm{dist}(\cdot,o)$ and let $$g=d\rho \otimes d\rho +\psi^2(\rho) g_{S^{N-1}}$$ for some smooth $\psi:(0,R) \rightarrow (0,\infty)$. Then the inequality $$\int_{M} |\partial_\rho \varphi|^p \omega_g \geq \bigg( \frac{p-1}{p} \bigg)^p \int_{M} \frac{|\varphi|^p}{(\varpi_p^\textrm{in})^p} \omega_g,\ \ \ \varphi \in C^1_c(M\setminus \{ o \})$$ is valid whenever $$\int_0^\rho \psi^{-\frac{N-1}{p-1}}(\xi)d\xi < \infty.$$
\end{theorem}

\begin{proof}
It is well known that in this case we have $\rho^{-1}(R)= \{ o' \}$ where $o'$ is a single point antipodal to $o$. It follows that $M\setminus \{ o,o' \}$ can be covered with polar coordinates in which the metric is expressed exactly as in the statement of the theorem. The rest of the proof is a repetition of the steps in the proof of Theorem \ref{Th2}, the only difference being applying $(\star \star)$ instead of $(\star)$.
\end{proof}

Of special interest are the cases of the $N$-Sphere $\mathbb{S}^N$, where $\psi(\theta)=\sin(\theta)$, $\theta\in (0,\pi)$, and the Hyperbolic Space $\mathbb{H}^N$, where $\psi(\rho) = \sinh(\rho)$, $\rho\in (0,\infty)$.

\begin{remark}
It recently came to our attention that this is not the first time that results such as these make their appearance. Other authors have employed analytic methods to obtain such results in a number of cases. For example, in \cite{CR}, the authors present some results for spheres and spherically symmetric domains that are very similar to our own. In \cite{BAGG}, the authors use a general result from \cite{DD} to derive an $L^p$ Hardy potential for the hyperbolic space that also has the same form as the one that occurs from our method. More generally, in the spherically symmertic case, the Hardy potentials that we are looking at are all of the form $|\nabla \rho|^p/\rho^p$ for some $p$-harmonic $\rho \in W^{1,p}(M)$, and can therefore be considered a special case of the main result in \cite{DD}.

Regardless, our method is inherently geometric instead of analytic and applies more generally, for example $X$ and $\omega$ need not be related by a Riemannian metric. Moreover, the potentials provided by our method are explicit in any case, symmetric or not.
\end{remark}

\section{Application III: The exterior of a black hole}

As a final application, we would like to discuss the case of the Schwarzschild metric, which describes static black holes in the context of General Relativity. The full Schwarzschild metric in (3+1)-dimensional spacetime reads $$-\bigg(1- \frac{1}{r} \bigg)dt \otimes dt + \bigg(1- \frac{1}{r} \bigg)^{-1} dr \otimes dr + r^2 g_{\mathbb{S}^2}$$ and is actually a pseudo-Riemannian metric. To get a Riemannian metric, we will simply restrict our attention on ``temporal slices" of constant time, where the restricted metric reads $$\bigg(1- \frac{1}{r} \bigg)^{-1} dr \otimes dr + r^2 g_{\mathbb{S}^2}.$$

\begin{theorem}[Hardy Inequality for the Schwarzschild Black Hole]
Let $\mathfrak{B} = \{ x\in \mathbb{R}^3 : |x|>1$ \} be equipped with the metric $$g_{\mathfrak{B}} = \frac{r}{r-1} dr \otimes dr + r^2 g_{\mathbb{S}^2}$$ as above, let $\nabla_\mathfrak{B}$ and $\omega_\mathfrak{B}$ stand for the Riemannian gradient and volume form, respectively, and let
\[ 
\delta= \left\{
\begin{array}{ll}
      2r^2\sqrt{\frac{r-1}{r}} & 1<r<(4/3) \\
      2r^2 \bigg( 1-\sqrt{\frac{r-1}{r}} \bigg) & r \geq (4/3)\\
\end{array} 
\right. .
\]
Then the inequality $$\int_{\mathfrak{B}} |\nabla_{\mathfrak{B}} \varphi |^2 \omega_{\mathfrak{B}} \geq \int_\mathfrak{B} \frac{r-1}{r} \bigg| \frac{\partial \varphi}{\partial r} \bigg|^2 \omega_\mathfrak{B}  \geq \frac{1}{4} \int_{\mathfrak{B}} \frac{|\varphi|^2}{\delta^2} \omega_{\mathfrak{B}}$$ is valid for all $\varphi \in C^1_c(\mathfrak{B})$.
\end{theorem}

\begin{proof}
Let $X=\sqrt{\frac{r-1}{r}} \frac{\partial}{\partial r}$. In polar coordinates we have $$\omega_\mathfrak{B}= \sqrt{\frac{r}{r-1}}r^2 \sin(\theta) dr \wedge d\theta \wedge d\phi.$$ We are looking for a new coordinate $t$ to replace $r$ such that $\partial/\partial t = X$. Let $f:(1,\infty) \rightarrow (0,\infty)$ be the function given by the formula $$f(x) = \sqrt{x}\sqrt{x-1}+\log(\sqrt{x}+\sqrt{x-1}).$$ It is easy to verify that $t=f(r)$ satisfies the imposed condition, therefore $(t,\theta,\phi)$ is a set of normal coordinates for $(\mathfrak{B},X)$. As $f$ is a bijection, let $g$ denote its inverse. Substituting $r=g(t)$ into the formula for $\omega_\mathfrak{B}$, we get $$\omega_\mathfrak{B}=g(t)^2 \sin(\theta) dt \wedge d\theta \wedge d\phi,$$ therefore $\Omega(t,\theta,\phi) = g(t)^2\sin(\theta)$. The temporal/volumetric distance in this case is $$\tau_2 = g(t)^2 \int_{\{0,\infty \}}^t g(w)^{-2} dw = r^2 \min \bigg( \int_0^t g(w)^{-2} dw,\int_t^\infty g(w)^{-2} dw \bigg).$$ Substituting $w=f(\xi)$, it is elementary to show that $$\tau_2= r^2 \int_{\{1,\infty\}}^r \frac{d\xi}{\xi^{3/2}(\xi-1)^{1/2}} = \delta$$ and the proof is complete.
\end{proof}

A more complete treatment of this matter will be given elsewhere.

\section{Higher-order inequalities}
Likewise, one can recursively obtain inequalities for higher order  differential operators. For example, consider the second-order operator $YX$ obtained by the composition of two directional derivatives (vector fields) $X,Y \in \Gamma(TM)$. If $(M,Y,\omega)$ is $p$-traceable, we obtain $$\int_M |YX \varphi|^p \omega \geq \bigg( \frac{p-1}{p} \bigg)^p \int_M \frac{|X\varphi|^p}{(\tau^Y_p)^p} \omega = \int_M \bigg| \frac{X}{\tau^Y_p} \varphi \bigg|^p \omega,$$ where $\tau^Y_p$ is the temporal/volumetric distance of $(M,Y,\omega)$. In the same manner, if $(M,X/\tau^Y_p,\omega)$ is $p$-traceable, we may repeat the process and obtain $$\int_M |YX \varphi|^p \omega \geq \bigg( \frac{p-1}{p} \bigg)^{2p} \int_M \frac{|\varphi|^p}{(\tau^{Y/\tau^X_p}_p)^p} \omega,$$ where $\tau^{Y/\tau^X_p}_p$ is the temporal/volumetric distance for $(M,X/\tau^Y_p,\omega)$. By induction, this process can produce inequalities for operators of the form $X_1 \cdots X_k$ for any $k\in \mathbb{N}$, provided that $p$-traceability holds for each step.

We give some examples of higher-order inequalities obtained in this way.

\begin{example}
 Recursive application of the weighted inequality of Example 5.3 yields the $k$-th order Rellich inequality $$\int_{\mathbb{R}^N} \bigg| \frac{\partial^k \varphi}{\partial r^k} \bigg|^p dx \geq \prod_{l=1}^{k} \bigg| \frac{lp-N}{p} \bigg|^p \int_{\mathbb{R}^N} \frac{|\varphi|^p}{r^{kp}} dx, \ \ \ \varphi \in C^1_c(\mathbb{R}^N \setminus \{ 0 \}).$$ Note that, in essence, if one has weighted inequalities for the vector fields of interest, computing the distance at each step becomes unnecessary.
 
 Likewise, for the one-dimensional case we have $$\int_{\mathbb{R}_+} |D^k \varphi|^p dx \geq \prod_{l=1}^{k} \bigg( \frac{lp-1}{p} \bigg)^p \int_{\mathbb{R}_+} \frac{|\varphi|^p}{x^{kp}} dx, \ \ \ \varphi \in C^1_c(\mathbb{R}_+),$$ which can be further integrated to give the same inequality for the half-space.
\end{example}

\begin{example}
Consider the second order differential operator $$H= \frac{1}{r} \frac{\partial}{\partial r} r \frac{\partial}{\partial r}.$$ Applying the weighted inequality of Example 5.3 twice yields the inequality $$\int_{\mathbb{R}^N} |H\varphi|^p dx \geq \bigg| \frac{2p-N}{p} \bigg|^{2p} \int_{\mathbb{R}^N} \frac{|\varphi|^p}{r^{2p}} dx, \ \ \ \varphi \in C^1_c(\mathbb{R}^N \setminus \{ 0 \}).$$
\end{example}

As a final interesting application, we will use the above to obtain Rellich inequalities involving the wave operator in the 2-dimensional half-space, which, in contrast to most operators that are being discussed in literature, is not an elliptic operator. We are not aware of other results of this type so far. We prove the following.

\begin{theorem}[Higher-order Rellich Inequality for the Wave Operator]
Let $\square = \partial_x^2 - \partial_y^2$ denote the 2-dimensional wave operator, and let $u \in C^\infty_c (\mathbb{R}^2_+)$. Then the inequality $$\int_{\mathbb{R}^2_+} |\square^k u|^p dxdy \geq \prod_{l=1}^{2k} \bigg( \frac{lp-1}{p} \bigg)^p \int_{\mathbb{R}^2_+} \frac{|u|^p}{y^{2kp}} dxdy$$ holds for all $k \in \mathbb{N}$. The constant is sharp.
\end{theorem}

This is an easy corollary of the following lemma.

\begin{lemma}
Let $u\in C^\infty_c (\mathbb{R}^2_+)$. Then the inequality $$\int_{\mathbb{R}^2_+} \frac{|(\partial_x \pm \partial_y)u|^p}{y^\gamma} dxdy \geq \bigg( \frac{\gamma +p -1}{p} \bigg)^p \int_{\mathbb{R}^2_+} \frac{|u|^p}{y^{\gamma+p}} dxdy$$ holds for all $\gamma > 1-p$.
\end{lemma}

\begin{proof}
Consider the case of $X:=(\partial_x+\partial_y)$. The coordinates $$t=\frac{1}{2}(x+y),\ s=\frac{1}{2}(y-x)$$ are a set of normal coordinates for $(\mathbb{R}^2_+,X)$ (it can be easily verified that $X=\partial/\partial t$). Moreover, we have that $x=t-s$ and $y=t+s$, thus $$dx=dt-ds,\ dy=dt+ds.$$ It follows that $dx\wedge dy = 2 dt\wedge ds$. It follows that $$\omega = \frac{1}{y^\gamma} dx\wedge dy=\frac{2}{(t+s)^\gamma} dt\wedge ds,$$ and the corresponding temporal/volumetric distance is $$\tau_p=(t+s)^{-\frac{\gamma}{p-1}} \int_{\partial I_s}^t (\xi+s)^{\frac{\gamma}{p-1}} d\xi,$$ where $I_s=(-s,\infty)$. By elementary calculations, this is equal to $$\tau_p = \frac{p-1}{\gamma+p-1} (t+s) = \frac{p-1}{\gamma+p-1} y$$ and the result follows.

The case of $(\partial_x-\partial_y)$ is entirely analogous.
\end{proof}

The inequality in the theorem follows from the fact that $\square = (\partial_x+\partial_y)(\partial_x-\partial_y)$ and inductive application of the lemma. Sharpness is proved by a standard argument, substituting the sequence $$u_\epsilon (x,y)= y^{\frac{2kp-1}{p}+\epsilon} \rho_\epsilon (x,y),\ \ \ \epsilon \rightarrow 0,$$ where $\rho_\epsilon$ is a suitable cutoff function that is equal to $\rho_\epsilon=1$ in $(-\epsilon,\epsilon) \times (\epsilon,1/\epsilon)$ and $\textrm{supp}(\rho_\epsilon) \subset (-2 \epsilon,2 \epsilon) \times (\epsilon/2,2/\epsilon)$.

\appendix
\section{Auxiliary Material}

We give some auxiliary results from the theory of differentiable manifolds that are used throughout our work. All of them can be found in \cite{L}.

Let $F:M \rightarrow N$ be a smooth map between manifolds. As usual, the differential of $F$ is defined to be the map $F_*: TM \rightarrow TN$ such that $F_*X[g]=X[g \circ F]$ for all $g\in C^\infty(N)$. Likewise, we define the pull-back of $F$ as the map $F^*: \Lambda(T^*N) \rightarrow \Lambda(T^*M)$ by $F^*\omega(X_1,\ldots,X_k) = \omega(F_* X_1,\ldots,F_* X_k)$ for all vectors $X_1,\ldots,X_k \in T_zM$ for all $z\in M$.

\begin{lemma}[Diffeomorphic invariance of the integral]
Let $F:N \rightarrow M$ be an orientation-preserving diffeomorphism and $\omega \in \Lambda^{top}(T^*M)$. Then $$\int_M \omega = \int_N F^* \omega.$$
\end{lemma}

\begin{lemma}[Integration over parametrisations]
Let $M$ be an oriented manifold of dimension $N$ and let $\omega \in \Lambda^N(TM)$ be a compactly supported top-form on $M$. Suppose $D_1,\ldots,D_k$ are open domains of integration in $\mathbb{R}^N$, and for $i=1,\ldots,k$ we are given smooth maps $\zeta_i: \bar{D}_i \rightarrow M$ satisfying
\begin{enumerate}
\item $\zeta_i$ restricts to an orientation-preserving diffeomorphism from $\bar{D}_i$ onto an open set $W_i \subset M$.
\item $W_i \cap W_j = \varnothing$ for $i \neq j$.
\item $\textrm{supp}(\omega) \subset \bar{W}_1 \cup \cdots \cup \bar{W}_k$.
\end{enumerate}
Then $$\int_M \omega = \sum_{i=1}^k \int_{D_i} \zeta_i^* \omega.$$
\end{lemma}

\paragraph*{\textbf{Acknowledgement.}} Special thanks are owed to my PhD supervisor, Professor G. Barbatis, for the time he spent reviewing the article and offering useful suggestions. This research was supported by the Hellenic Foundation for Research and Innovation (HFRI) under the HFRI PhD Fellowship grant (Fellowship Number 1250).



\begin{thebibliography}{20}

\bibitem {AM} Avkhadiev, F.G., Makarov, R.V. Hardy Type Inequalities on Domains with Convex Complement and Uncertainty Principle of Heisenberg. Lobachevskii J Math 40, 1250–1259 (2019).

\bibitem {BEL} Balinsky, Alexander A, Evans, W Desmond, Lewis, Roger T, The Analysis and Geometry of Hardy's Inequality (2015), Springer International Publishing.

\bibitem {BFT} Barbatis, Gerassimos \& Filippas, Stathis \& Tertikas, Achilles. (2003). Tertikas A unified approach to improved L p Hardy inequalities with best constants. Transactions of the American Mathematical Society.

\bibitem {BAGG} Elvise Berchio, Lorenzo D'Ambrosio, Debdip Ganguly, Gabriele Grillo,
Improved Lp-Poincaré inequalities on the hyperbolic space, Nonlinear Analysis, Volume 157, 2017, Pages 146-166.

\bibitem {CR} F. Chiacchio and T. Ricciardi, Some sharp Hardy inequalities on spherically symmetric domains. Pac. J. Math. 242 No. 1 (2009), 173-187.

\bibitem {DD} Lorenzo D'Ambrosio, Serena Dipierro, Hardy inequalities on Riemannian manifolds and applications, Annales de l'Institut Henri Poincare (C) Non Linear Analysis, Volume 31, Issue 3, 2014, Pages 449-475.

\bibitem {GPP} D. Goel, Y. Pinchover, and G. Psaradkis, On weighted Lp-Hardy inequality on domains in Rn, to appear in a special issue dedicated to Shmuel Agmon, Pure Appl. Funct. Anal., arXiv: 2012.12860

\bibitem {KO} Kombe, Ismail, and Murad Özaydin. “Improved Hardy and Rellich inequalities on Riemannian manifolds.” Transactions of the American Mathematical Society, vol. 361, no. 12, 2009, pp. 6191–6203. JSTOR, www.jstor.org/stable/40590795. Accessed 19 Apr. 2021.

\bibitem {K} Alexandru Krist\'aly, Sharp uncertainty principles on Riemannian manifolds: the influence of curvature, Journal de Mathématiques Pures et Appliquées, Volume 119, 2018, Pages 326-346.

\bibitem {L} Lee, John M., Introduction to Smooth Manifolds, Second Edition, Springer Science+Business Media New York, 2003.

\bibitem {MMP} M. Marcus, V. J. Mizel and Y. Pinchover, On the best constant for Hardy’s inequality in $\mathbb{R}^n$, Trans. Amer. Math. Soc. 350 (1998), 3237-3255.

\bibitem {SP} Sun, X., Pan, F. Hardy type inequalities on the sphere. J Inequal Appl 2017, 148 (2017).

\end{thebibliography}
\end{document}